\documentclass[a4paper, 11pt]{amsart}

\usepackage{microtype}
\usepackage{amsthm,amsfonts,amsmath,amssymb}
\usepackage[alphabetic,initials,nobysame]{amsrefs}
\usepackage{mathtools} 
\usepackage{pxfonts}
\usepackage{euscript}
\usepackage[utf8]{inputenc}
\usepackage{bbold,bbm}
\usepackage{hyperref}
\hypersetup{
    colorlinks,
    linkcolor={red!50!black},
    citecolor={blue!50!black},
    urlcolor={blue!80!black}
}
\usepackage{graphics}
\usepackage{epstopdf} 
\usepackage{float}
\usepackage{graphicx}
\usepackage{comment}
\usepackage[most]{tcolorbox}
\usepackage{scalerel}
\usepackage[all,2cell,arrow,matrix,tips]{xy} \UseAllTwocells \SilentMatrices

\newtheorem{thm}{Theorem}[section]
\newtheorem{lem}[thm]{Lemma}
\newtheorem{prop}[thm]{Proposition}
\newtheorem{cor}[thm]{Corollary}

\newtheorem{prop-defn}[thm]{Proposition-Definition}

\theoremstyle{definition}
\newtheorem{defn}[thm]{Definition}

\newtheorem{exam}[thm]{Example}
\newtheorem{rem}[thm]{Remark}

\theoremstyle{remark}

\newcommand\Cb {\mathbb{C}}

\newcommand\BB {\mathbf{B}}

\newcommand\CC {\EuScript{C}}
\newcommand\CD {\EuScript{D}}

\newcommand\CM {\EuScript{M}}

\newcommand{\one}{\mathbb{1}}

\DeclareMathOperator{\Id}{Id}
\DeclareMathOperator{\Tr}{Tr}
\DeclareMathOperator{\Fun}{Fun}
\DeclareMathOperator{\bLMod}{\mathbf{LMod}}

\DeclareMathOperator{\bRMod}{\mathbf{RMod}}
\DeclareMathOperator{\sbRMod}{\mathbf{sRMod}}
\DeclareMathOperator{\bBMod}{\mathbf{BiMod}}

\DeclareMathOperator{\vecc}{\mathbf{Vec_{\textrm{f.d.},\mathbb{C}}}}
\DeclareMathOperator{\hilbc}{\mathbf{Hilb_{\textrm{f.d.},\mathbb{C}}}}

\DeclareMathOperator{\ev}{ev}
\DeclareMathOperator{\op}{op}

\newcommand{\mref}[1]{(\ref{#1})}
\newcommand{\vertiii}[1]{{\left\vert\kern-0.25ex\left\vert\kern-0.25ex\left\vert #1 
    \right\vert\kern-0.25ex\right\vert\kern-0.25ex\right\vert}}
\DeclarePairedDelimiterX\braket[2]{\langle}{\rangle}{#1 \delimsize\vert #2}

\begin{document}

\title{Separable algebras in multitensor C$^*$-categories are unitarizable}

\author{Luca Giorgetti$^{1}$, Wei Yuan$^{2,3}$, XuRui Zhao$^{4}$}

\address{1 Dipartimento di Matematica, Universit\`a di Roma Tor Vergata\\
Via della Ricerca Scientifica, 1, I-00133 Roma, Italy}
\email{giorgett@mat.uniroma2.it}

\address{2 Institute of Mathematics, Academy of Mathematics and Systems Science \\
Chinese Academy of Sciences, Beijing, 100190, China}

\address{3 School of Mathematical Sciences, University of Chinese Academy of Sciences \\
Beijing 100049, China}
\email{wyuan@math.ac.cn}

\address{4 School of Mathematical Sciences, University of Chinese Academy of Sciences \\
Beijing 100049, China}
\email{zhaoxurui16@mails.ucas.ac.cn}

\begin{abstract}
   Recently, S. Carpi et al. (Comm. Math. Phys., 402:169--212, 2023) proved that every connected (i.e. haploid) Frobenius algebra in a tensor C$^*$-category is unitarizable (i.e. isomorphic to a special C$^*$-Frobenius algebra).
   Building on this result, we extend it to the non-connected case by showing that an algebra in a multitensor C$^*$-category is unitarizable if and only if it is separable. 
\end{abstract}

\subjclass[2020]{} 
\thanks{Research by the first author (L.G.) is supported by \lq\lq MIUR Excellence Department Project MatMod@TOV" awarded to the Department of Mathematics, University of Rome Tor Vergata, CUP E83C23000330006, and by the University of Rome Tor Vergata funding OAQM, CUP E83C22001800005, by progetto GNAMPA 2023 \lq\lq Metodi di Algebre di Operatori in Probabilit\`a non Commutativa" CUP E53C22001930001, and by progetto GNAMPA 2024 \lq\lq Probabilit\`a Quantistica e Applicazioni" CUP E53C23001670001. Research by the second author (W.Y.) is supported by the NSFC under grant numbers 11971463, 11871303, 11871127.}

\maketitle

\section{Introduction} 

Separable algebras in tensor categories are a natural generalization of finite-dimensional (associative unital) semisimple algebras over $\mathbb{C}$. Let $\CC$ be a tensor category, see e.g. \cite{Mue10tour}, \cite{EGNO15}. If $\CC$ happens to be in addition unitary i.e. C$^*$, see e.g. \cite{NeTuBook}, \cite{BKLR15}, the main result of this note, Theorem \ref{thm:isomQsys_iff_sep}, states that every separable algebra is \lq\lq unitarizable" i.e. it is isomorphic to a \lq\lq unitarily\rq\rq\ separable algebra, and the converse holds trivially. For the precise notions see Definition \ref{def:alg_CstarFrob_special}, Definition \ref{def:alg_separable}, and Definition \ref{def:specialCstarFrob_unitarilysep_alg}. By Theorem \ref{thm:isomQsys_iff_sep}, every statement involving separable algebras living in a tensor or multitensor C*-category has a \lq\lq unitary" counterpart.

On the one hand, unitarily separable algebras also appear in the literature under the name of special C$^*$-Frobenius algebras \cite{BKLR15} or Q-systems \cite{Lon90}, \cite{Lon94}, \cite{LoRo97}. Their study was initially motivated by the applications to operator algebras, in particular to the construction and classification of finite-index subfactors \cite{Jon83}, \cite{Ocn88}, \cite{Pop90}, \cite{Pop95}, \cite{JonPA}. See \cite{EvKaBook} for an introduction to the subject, \cite{Gio22} for an overview, and \cite{AMP23}, and references therein, for recent classification results. Since \cite{LoRe95}, Q-systems also play a pivotal role in the construction and classification of finite-index extensions of algebraic quantum field theories \cite{HaagBook} in arbitrary spacetime dimensions, and of one-dimensional conformal field theories in the (completely) unitary vertex operator algebra framework \cite{KacBook}, \cite{CKLW18} as well, since \cite{Gui22}. Recently, Q-systems have been employed in the study of \lq\lq quantum symmetries" (tensor category actions, generalizing ordinary group symmetries) of C$^*$-algebras \cite{CHPJP22}, \cite{ChPe22}, \cite{CHPJ24}, \cite{EvGP23}.

On the other hand, separable algebras have a priori no inbuilt unitarity. Together with an additional commutativity assumption with respect to a given braiding, since \cite{DMNO13} they are also often called \'etale algebras. These objects, typically assuming connectedness, are studied in relation to Ocneanu's quantum subgroups \cite{Ocn02}. See \cite{Gan23-arxiv} for recent results and a detailed account on their classification program. As for (commutative irreducible) Q-systems in the algebraic quantum field theory framework, connected \'etale algebras can be used to describe (local irreducible) extensions of vertex operator algebras \cite{HKL15}, see also \cite{KO02}, \cite{CKM17-arxiv}. Notably, they describe all rational 2D conformal field theories maximally extending a given tensor product of (isomorphic) chiral subtheories. See \cite{FRS02}, \cite{FRS04}, \cite{FRS04II}, \cite{FRS05}, \cite{RFFS07} in the Euclidean setting, \cite{HuKo07}, \cite{Kon07} in the full vertex operator algebra setting, \cite{BKL15}, \cite{BKLR16} for the algebraic quantum field theory setting, and \cite{AGT23-online} for the Wightman quantum field theory setting. See also \cite{KYZ21} for a proof of functoriality of the \cite{FRS02} construction when varying the given chiral subtheory.

The proof of our main result, Theorem \ref{thm:isomQsys_iff_sep}, strongly relies on Theorem 3.2 in \cite{CGGH23}. In the connected (i.e. haploid) case, the notions of separable algebra, Frobenius algebra, and isomorphic to unitarily separable algebra (i.e. isomorphic to special C$^*$-Frobenius algebra = Q-system) all coincide by Lemma \ref{lem:con_sep_alg_Frob} below and by Theorem 3.2, see also Remark 3.3, in \cite{CGGH23}. In the non-connected case, we first decompose a separable algebra $A$ in $\CC$ into indecomposable ones, Lemma \ref{lem:sep_alg_sum_ind}, then unitarize the category of right $A$-modules in $\CC$, Lemma \ref{lem:sep_alg_unitary_module}, lastly we show that the unitarized category is equivalent to the modules over a unitarily separable algebra in $\CC$ to which $A$ is isomorphic, Proposition \ref{prop:internal_C_*_Frob_alg}. This leads to Theorem \ref{thm:isomQsys_iff_sep}

We point out that the semisimplicity of $\CC$ (or better of the tensor subcategory generated by $A$) is a consequence of the assumptions made in Theorem 3.2 in \cite{CGGH23} (where the tensor unit of $\CC$ is simple). Here, we need semisimplicity of $\CC$ to exploit the separability of $A$ via Proposition \ref{prop:sep_alg_bimod_semi}. Thus, a possible generalization of Theorem \ref{thm:isomQsys_iff_sep} to the case of non-semisimple monoidal C$^*$-categories $\CC$ should require a different idea, possibly \lq\lq internal" to the C$^*$-algebra $\CC(A,A)$, on how to show directly that a separable algebra is isomorphic in $\CC$ to a unitarily separable one.

\section{Preliminaries}

A \textbf{C$^*$-category} is a generalization of a C$^*$-algebra of operators acting between different Hilbert spaces instead of one. The objects $X,Y,Z, \ldots$ of $\CC$ can be thought of as the Hilbert spaces, the morphisms $f,g,h,\ldots$ of $\CC$ as the bounded linear operators. Formally, it is a $\Cb$-linear category $\CC$ (\cite{Mac98}, \cite{EGNO15}) equipped with an involutive contravariant anti-linear endofunctor $*: \CC \to \CC$ (sometimes called \textit{dagger} or \textit{adjoint}) and a family of norms $\|\cdot\|$ on morphisms such that
\begin{itemize}
    \item the endofunctor $*$ is the identity on objects (we use $f^*\in\CC(Y,X)$ to denote the image of the morphism $f\in\CC(X,Y)$),
    \item the hom space $\CC(X,Y)$ is a Banach space for every $X, Y \in \CC$, 
    \item $\|gf\| \leq \|g\|\|f\|$, $\|f^*f\| = \|f\|^2$, $f^*f \geq 0$, for every $f \in \CC(X, Y)$, $g \in \CC(Y, Z)$.
\end{itemize}
In particular, a C$^*$-category with one object is a unital C$^*$-algebra (see \cite{GLR85}).

In the following, we use $1_X$ to denote the identity morphism in $\CC(X, X)$. For a morphism $f\in\CC(X,Y)$ we will occasionally write $f:X\to Y$ if the environment category $\CC$ is clear from the context.

A morphism $f$ in a C$^*$-category is called \textit{unitary} (resp. \textit{self-adjoint}) if $f^* = f^{-1}$ (resp. $f^* = f$). Let $\CC$ and $\CD$ be two C$^*$-categories. A \textbf{$*$-functor} from $\CC$ to $\CD$ is a linear functor such that $F(f^*) = F(f)^*$ for every morphism $f$. 

A \textbf{multitensor C$^*$-category} is an abelian rigid (\cite{DoRo89}, \cite{LoRo97}) monoidal category $(\CC, \otimes: \CC \times \CC \to \CC, \one)$ equipped a C$^*$-category structure satisfying the following conditions
\begin{itemize}
    \item the tensor unit $\one$ of $\CC$ is semisimple, i.e. $\CC(\one,\one)$ is finite-dimensional,  
    \item $\otimes$ is a bilinear functor and $(f \otimes g)^* = f^* \otimes g^*$ for every morphisms $f, g$,
    \item the associator and the left/right unitor constraints are unitary.
\end{itemize}
If $\CC(\one, \one) \simeq \Cb$, i.e. if $\one$ is simple, then $\CC$ is called a \textbf{tensor C$^*$-category}. By Proposition 8.16 in \cite{GL19}, every multitensor C$^*$-category $\CC$ is semisimple and locally finite. Moreover, by Mac Lane's coherence theorem, $\CC$ is equivalent to a strict multitensor C$^*$-category, i.e. where the associator and the left/right unitors are identities (see \cite{EGNO15} and \cite{BKLR15}).
From now on, unless otherwise specified, we use $\CC$ to denote a (strict) multitensor C$^*$-category.

\begin{rem}\label{rem:standard_sol}
    The tensor unit $\one$ of $\CC$ is a direct sum of simple objects $\oplus_{i=1}^n \one_i$. Note that $\CC \simeq \oplus_{ij} \CC_{ij}$, where $\CC_{ij} := \one_i \otimes \CC \otimes \one_j$ (see Remark 4.3.4 in \cite{EGNO15}). Let $\tau$ be the linear functional on $\CC(\one, \one)$ defined by 
\begin{align*}
    \tau \left (\sum_{i} a_i 1_{\one_i} \right) := \sum_i a_i.   
\end{align*}
 
    Let $X \in \CC$. We have $X \simeq \oplus_{ij} X_{ij}$ and $\overline{X} \simeq \oplus_{ij} \overline{X}_{ji}$, where $X_{ij}:= \one_i \otimes X \otimes \one_j$ and $\overline{X}$, $\overline{X}_{ij}$ denote the dual (or conjugate) objects of $X$, $X_{ij}$ respectively. Namely, for every $i,j\in\{1,\ldots,n\}$, there exists (see below) a solution $(\gamma_{ij} \in \CC(\one_j , \overline{X}_{ij} \otimes X_{ij}), \overline{\gamma}_{ij} \in \CC(\one_i , X_{ij} \otimes \overline{X}_{ij}))$ of the conjugate equations 
\begin{align*}
    (\overline{\gamma}_{ij}^{\,*} \otimes 1_{{X}_{ij}}) (1_{{X}_{ij}} \otimes {\gamma_{ij}}) = 1_{{X}_{ij}}, \quad ({\gamma}_{ij}^{*} \otimes 1_{\overline{X}_{ij}}) (1_{\overline{X}_{ij}} \otimes {\overline{\gamma}_{ij}}) = 1_{\overline{X}_{ij}},
\end{align*}
which is unique up to unitaries, and such that
\begin{align}\label{equ:standard_sol_1}
    \tau \left (\gamma_{ij}^*(1_{\overline{X}_{ij}} \otimes f) \gamma_{ij} \right) = \tau \left (\overline{\gamma}^{\,*}_{ij}(f \otimes 1_{\overline{X}_{ij}}) \overline{\gamma}_{ij} \right )
\end{align}
for every $f \in \CC(X_{ij}, X_{ij})$. The \textit{scalar dimension} of $X_{ij}$ (\cite{LoRo97}, \cite{GL19}) is then $d_{X_{ij}} = \tau(\gamma_{ij}^*\gamma_{ij}) = \tau(\overline{\gamma}_{ij}^{\,*}\overline{\gamma}_{ij})$. 

    For the convenience of the reader, we sketch proof of this well-known fact when $i \neq j$ (the case where $i = j$ can be proved similarly). Let $\{Z_s\}_s$ be a set of representatives of simple objects in $\CC_{ij}$. Since $\dim \CC(\one_j, \overline{Z}_{s} \otimes Z_{s}) = \dim \CC(\one_i, Z_{s} \otimes \overline{Z}_{s}) = 1$, we can choose a solution of the conjugate equations $(\gamma_{s}, \overline{\gamma}_{s})$ such that $\tau(\gamma_{s}^*\gamma_{s}) = \tau(\overline{\gamma}_{s}^{\,*}\overline{\gamma}_{s})$, i.e. $\|\gamma_{s}\| = \|\overline{\gamma_{s}}\|$ (as in Definition 3.4 in \cite{LoRo97}). For non-simple $X_{ij} \in \CC_{ij}$, let $\{u_{s,k}\}_k$ (resp. $\{\overline{u}_{s, k}\}_k$) be a basis of $\CC(Z_s, X_{ij})$ (resp. $\CC(\overline{Z}_s, \overline{X}_{ij})$) such that $u_{s, l}^* u_{s,k} = \delta_{k, l} 1_{Z_s}$ (resp. $\overline{u}_{s, l}^{\,*} \overline{u}_{s,k} = \delta_{k, l} 1_{\overline{Z}_s}$). Let 
\begin{align*}
    \gamma_{ij} := \sum_{s} \sum_{k} (\overline{u}_{s,k} \otimes u_{s, k}) \gamma_{s}, \quad \overline{\gamma}_{ij} := \sum_{s} \sum_{k} (u_{s, k} \otimes \overline{u}_{s,k}) \overline{\gamma}_{s},
\end{align*}
as before Lemma 3.7 in \cite{LoRo97}, or before Lemma 8.23 in \cite{GL19}, then $(\gamma_{ij}, \overline{\gamma}_{ij})$ is a solution of the conjugate equations that satisfies the equation \mref{equ:standard_sol_1}. Indeed, 
\begin{align*}
    \tau \left (\gamma_{ij}^*(1_{\overline{X}_{ij}} \otimes u_{s,k} u_{s, l}^*) \gamma_{ij} \right) = \delta_{k, l} \tau(\gamma_s^* \gamma_s) = \delta_{k, l} \tau(\overline{\gamma}_s^{\,*} \overline{\gamma}_s) = \tau \left (\overline{\gamma}^{\,*}_{ij}(u_{s,k} u_{s, l}^* \otimes 1_{\overline{X}_{ij}}) \overline{\gamma}_{ij} \right ).
\end{align*}
    Let $(\omega \in \CC(\one , \overline{X}_{ij} \otimes X_{ij}), \overline{\omega} \in \CC(\one , X_{ij} \otimes \overline{X}_{ij}))$ be a solution of the conjugate equations that satisfies the equation \mref{equ:standard_sol_1}. Then there exists an invertible morphism $h \in \CC(X_{ij}, X_{ij})$ such that $\omega = (1_{\overline{X}_{ij}} \otimes h) \gamma_{ij}$ and $\overline{\omega} = ((h^*)^{-1} \otimes 1_{\overline{X}_{ij}})\overline{\gamma}_{ij}$. By choosing a different basis of $\CC(Z_s, X_{ij})$, we may assume that $h = \sum_s \sum_k a_{s, k} u_{s, k}u_{s, k}^*$, where $a_{s,k} > 0$. Then the condition that $(\omega, \overline{\omega})$ fulfills the equation \mref{equ:standard_sol_1} implies that $h = 1_{X_{ij}}$. In other words, the solution of the conjugate equations that satisfies the equation \mref{equ:standard_sol_1} is unique up to unitaries (see Lemma 3.3 and Lemma 3.7 in \cite{LoRo97}, and cf. Lemma 8.35 in \cite{GL19}, for more details).

    Let $\gamma_X := \oplus_{ij} \gamma_{ij}$ and $\overline{\gamma}_X := \oplus_{ij} \overline{\gamma}_{ij}$. Note that these are not the \textit{standard solutions} of the conjugate equations defined in \cite{GL19}, where the Perron--Frobenius data of the \emph{matrix dimension} enter as numerical prefactors for each $i,j$ (see Definition 8.25 and Definition 8.29 therein), unless the tensor unit is simple (as in Section 3 of \cite{LoRo97}) and they coincide with the standard solutions of \cite{LoRo97}. In particular, the \lq\lq loop\rq\rq\ or \lq\lq bubble\rq\rq\ morphisms $\gamma_X^*\gamma_X$ and $\overline{\gamma}_X^{\,*} \overline{\gamma}_X$ will neither be scalar in $\CC(\one,\one)$, nor equal, nor will $(\gamma_X,\overline{\gamma}_X)$ be \emph{spherical} (resp. \emph{minimal}) in the sense of Theorem 8.39 (resp. Theorem 8.44) in \cite{GL19}.
    
With the $(\gamma_X, \overline{\gamma}_X)$ defined above, we have
\begin{align*}
    \left (\gamma_{Y}^* \otimes 1_{\overline{X}} \right) \left (1_{\overline{Y}} \otimes g \otimes 1_{\overline{X}} \right) \left (1_{\overline{Y}} \otimes \overline{\gamma}_{X} \right) = \left (1_{\overline{X}} \otimes \overline{\gamma}_{Y}^{\,*} \right) \left (1_{\overline{X}} \otimes g \otimes 1_{\overline{Y}} \right) \left (\gamma_{X} \otimes 1_{\overline{Y}} \right)
\end{align*}
and 
\begin{align*}
     \tau \left (\gamma_{X}^*(1_{\overline{X}} \otimes hg) \gamma_X \right) = \tau \left (\overline{\gamma}_{X}^{\,*}(hg \otimes 1_{\overline{X}}) \overline{\gamma}_X\right ) = \tau\left (\gamma_{Y}^*(1_{\overline{Y}} \otimes gh) \gamma_Y \right) 
\end{align*}
for every $g \in \CC(X, Y)$, $h \in \CC(Y, X)$, and $X, Y \in \CC$. Moreover, if a solution of the conjugate equations $(\omega \in \CC(\one , \overline{X} \otimes X), \overline{\omega} \in \CC(\one , X \otimes \overline{X}))$ fulfills
\begin{align*}
     \tau \left (\omega^*(1_{\overline{X}} \otimes g \right ) \omega) = \tau \left (\overline{\omega}^{\,*}(g \otimes 1_{\overline{X}}) \overline{\omega} \right ), \quad \forall g \in \CC(X, X),
\end{align*}
then there exists a unitary $u \in \CC(X, X)$ (or $\overline{u} \in \CC(\overline{X}, \overline{X})$) such that $\omega = (1_{\overline{X}} \otimes u) \gamma_X$ and $\overline{\omega} = (u \otimes 1_{\overline{X}})\overline{\gamma}_X$ (or $\omega = (\overline{u} \otimes 1_X) \gamma_X$ and $\overline{\omega} = (1_X \otimes \overline{u})\overline{\gamma}_X$). 

    Based on these observations, it is not hard to check that $\CC$ endowed with the pivotal duality $\{(\overline{X}, \gamma_X, \overline{\gamma}_X)\}_{X \in \CC}$ is a \textit{pivotal category} (see, e.g. Section 1.7 in \cite{TV17} for the definition of pivotal category).
\end{rem}

\section{Algebras and modules in multitensor C$^*$-categories}

We recall below the natural generalization of the notion of finite-dimensional unital associative algebra (in the tensor category of finite-dimensional complex vector spaces $\vecc$). Let $\CC$ be a strict multitensor C$^*$-category.

\begin{defn}
    An \textbf{algebra in $\CC$} is a triple $(A, m, \iota)$, where $A$ is an object in $\CC$, $m \in \CC(A \otimes A , A)$ is the \lq\lq multiplication" morphism, $\iota \in \CC(\one , A)$ is the \lq\lq unit" morphism, fulfilling the associativity and unit laws
\begin{align*}
     m (m \otimes 1_A) = m (1_A \otimes m), \quad m(\iota \otimes 1_A) = 1_A = m(1_A \otimes \iota).
\end{align*}
\end{defn}

\begin{defn}
    Two algebras $(A, m, \iota)$ and $(A', m', \iota')$ in $\CC$ are said to be \textbf{isomorphic} if there is an invertible (not necessarily unitary) morphism $t\in\CC(A,A')$ such that $t m = m' (t \otimes t)$ and $t \iota = \iota'$.
\end{defn}

\begin{defn}\label{def:alg_CstarFrob_special}
    An algebra $(A, m, \iota)$ in $\CC$ is called a \textbf{C$^*$-Frobenius algebra} if $m^*$ is a left (or equivalently right) $A$-module morphism such that
\begin{align}\label{eq:Cstar_Frob}
     (m \otimes 1_A) (1_A \otimes m^*) = m^*m = (1_A \otimes m)(m^* \otimes 1_A).
\end{align}    
    An algebra $(A, m, \iota)$ in $\CC$ is called \textbf{special} if the multiplication is a coisometry:\footnote{or, in a different convention, a scalar multiple of a coisometry, cf. \cite{Mue03-I}, \cite{GS12}, \cite{BKLR15}, \cite{NY18}, \cite{ADC19}. Also, note that we do neither ask $\iota^*\iota$ to be 1, nor a multiple of 1, and that the latter condition is automatic if the tensor unit is simple.}
\begin{align*}
mm^* = 1_A.
\end{align*} 
\end{defn}


\begin{defn}
    Forgetting the C$^*$ structure, an algebra $(A, m, \iota)$ in $\CC$ endowed with a \textbf{coalgebra} structure $(A, \Delta\in\CC(A,A\otimes A), \varepsilon \in \CC(A,\one))$ (not necessarily $\Delta = m^*$, $\varepsilon = \iota^*$) fulfilling the coassociativity and counit laws, is called a \textbf{Frobenius algebra} if the analogue of \eqref{eq:Cstar_Frob} holds with $m^*$ replaced by $\Delta$ (see \cite{Abr99}, \cite{FRS02}, \cite{Yam04FrA}).
\end{defn}

The following crucial results proven in \cite{LoRo97}, \cite{FRS02}, \cite{BKLR15} assuming $\CC(\one,\one) \simeq \Cb$, see in particular Chapter 3 in \cite{BKLR15}, also hold for multitensor C$^*$-categories, cf. Section 2.2 in \cite{GiYu23}:

\begin{prop}\label{prop:algs_and_special_algs}
Let $(A, m, \iota)$ be an algebra in $\CC$. 
    \begin{itemize}
        \item If $(A, m, \iota)$ is special, then it is a C$^*$-Frobenius algebra.
        \item If $(A, m, \iota)$ is a C$^*$-Frobenius algebra, then it is isomorphic to a special one.
    \end{itemize}
\end{prop}

\begin{exam}
Recall, e.g. from Section 2 in \cite{Abr99} and Section 2.1 in \cite{NY18}, that a C$^*$-Frobenius algebra in $\hilbc$, the tensor C$^*$-category of finite-dimensional Hilbert spaces, is just an ordinary finite-dimensional C$^*$-algebra with a Frobenius structure. Forgetting the C$^*$ structure, a Frobenius algebra in the tensor category $\vecc$ of finite-dimensional vector spaces is a finite-dimensional Frobenius algebra.
\end{exam}

We shall use module categories (and their unitary version, C$^*$-module categories recalled below) over multitensor C$^*$-categories. See \cite{Ost03} or Chapter 7 in \cite{EGNO15} for the definitions of module category over a monoidal category $\CC$ and module functor.

\begin{defn}
    A left \textbf{C$^*$-module category} over a multitensor C$^*$-category $\CC$ is a left $\CC$-module category $(\CM, \odot: \CC \times \CM \to \CM)$ which is also a C$^*$-category, such that  
\begin{itemize}
    \item $\odot$ is bilinear and $(f \odot g)^* = f^* \odot g^*$ for every morphisms $f \in \CC$, $g \in \CM$,
    \item the associator and the unitor constraints are unitary.
\end{itemize} 
    Right C$^*$-module categories and C$^*$-bimodule categories are defined similarly. 
\end{defn}

Typical examples of left (resp. right) $\CC$-module categories (not necessarily C$^*$) come from considering right (resp. left) modules over an algebra $(A, m, \iota)$ in $\CC$. We use $\bRMod_{\CC}(A)$ (resp. $\bLMod_{\CC}(A)$) to denote the category of right (resp. left) $A$-modules in $\CC$.

\begin{defn}
Let $(A, m, \iota)$ be a special C$^*$-Frobenius algebra in $\CC$. As for algebras, a right \textbf{$A$-module} $(X, r \in \CC(X\otimes A, X))$ in $\CC$ is called \textbf{special} if
\begin{align*}
rr^* = 1_X.
\end{align*}
We denote by $\sbRMod_{\CC}(A)$ the category of special right $A$-modules in $\CC$. The definition for left $A$-modules is analogous.
\end{defn}
  
By the arguments of Chapter 3 in \cite{BKLR15}, cf. Section 2.2 in \cite{GiYu23}, we have:
 
\begin{prop}\label{prop:special_Cstar_mod}
    Let $(A, m, \iota)$ be a special C$^*$-Frobenius algebra in $\CC$. Then $\sbRMod_{\CC}(A)$ is a left C$^*$-module category over $\CC$, where the involution and norms are inherited from $\CC$. 
    
More generally, given a right $A$-module $\left (X, r \in \CC(X \otimes A, X) \right)$, then $(X, r' := h^{-1} r (h \otimes 1_A))$ is a special right $A$-module, where $h := \sqrt{r r^*}$, and $h^{-1}$ is a right $A$-module isomorphism from $(X, r)$ to $(X, r')$. Moreover, $\bRMod_{\CC}(A)$ is a left C$^*$-module category over $\CC$ with the following C$^*$-structure
\begin{itemize}
    \item $f \in \bRMod_{\CC}(A)(X, Y) \mapsto h_X^2 f^* h_Y^{-2} \in \bRMod_{\CC}(A)(Y, X)$,
    \item $\vertiii{f}:= \left \|h_Y^{-1} f h_X \right \|$, $f \in \bRMod_{\CC}(A)(X, Y)$,
\end{itemize}
where $h_X := \sqrt{r_X r_X^*}$ and $h_Y := \sqrt{r_Y r_Y^*}$ are defined respectively from the right $A$-module actions of $X$ and $Y$.
The embedding $\sbRMod_{\CC}(A) \xhookrightarrow{} \bRMod_{\CC}(A)$ is an equivalence of left C$^*$-module categories.
\end{prop}

\section{Separable algebras are unitarizable}

In this section, we prove our main theorem.

\begin{defn}\label{def:alg_separable}
An algebra $(A, m, \iota)$ in $\CC$ is called \textbf{separable} if the multiplication $m \in \CC(A \otimes A, A)$ splits as a morphism of $A$-$A$-bimodules in $\CC$, i.e. if there is an $A$-$A$-bimodule morphism $f \in \CC(A, A \otimes A)$ such that $mf = 1_A$. 
\end{defn}

Clearly, every (not necessarily special) C$^*$-Frobenius algebra in $\CC$ is separable. Indeed, by Proposition \ref{prop:algs_and_special_algs}, it is isomorphic to a special algebra in $\CC$ (Definition \ref{def:alg_CstarFrob_special}), namely $mm^* = 1_A$ holds up to isomorphism of algebras, hence it is separable.

Moreover, a special C$^*$-Frobenius algebra, which is also called a \textbf{Q-system} after \cite{Lon94} (see also \cite{LoRo97}, \cite{Mue03-I}, \cite{BKLR15}, \cite{CHPJP22}, \cite{CGGH23} and references therein), can be viewed as a \lq\lq unitarily\rq\rq\ separable algebra. The following definition is motivated by this fact.

\begin{defn}\label{def:specialCstarFrob_unitarilysep_alg}
    A (Frobenius) algebra in $\CC$ is \textbf{unitarizable} if it is (not necessarily unitarily) isomorphic to a special C$^*$-Frobenius algebra in $\CC$. 
\end{defn}

Our main result (Theorem \ref{thm:isomQsys_iff_sep}) states that every separable algebra in $\CC$ is unitarizable. \\

By the proof of Proposition 7.8.30 in \cite{EGNO15}, cf. Section 3 in \cite{Ost03}, Section 2.3 in \cite{DMNO13}, Section 2.4 in \cite{HPT16}, Section 4 in \cite{KoZh17-arxiv}, the following characterization of separability for algebras in (not necessarily C$^*$) multitensor categories holds:

\begin{prop}\label{prop:sep_alg_bimod_semi}
Let $(A, m_A, \iota_A)$, $(B, m_B, \iota_B)$ be separable algebras in $\CC$. Then the categories $\bRMod_{\CC}(A)$, $\bLMod_{\CC}(A)$, and $\bBMod_{\CC}(A|B)$ ($A$-$B$-bimodules in $\CC$) are semisimple. 

In particular, an algebra $(C, m_C, \iota_C)$ in $\CC$ is separable if and only if $\bBMod_{\CC}(C|C)$ is semisimple.
\end{prop}

Let $(A, m, \iota)$ be an algebra in $\CC$, $(X, r) \in \bRMod_{\CC}(A)$, and $(Y,l) \in \bLMod_{\CC}(A)$. We recall, e.g. from Section 7.8 in \cite{EGNO15} \textit{tensor product} of $X$ and $Y$ over $A$ is the object $X \otimes_A Y \in \CC$ defined as the co-equalizer of the diagram
\begin{align*}
    \xymatrix{
        X \otimes A \otimes Y \ar@<+.5ex>[r]^{\hspace{2mm} r \otimes 1_Y} \ar@<-.5ex>[r]_{\hspace{2mm} 1_X \otimes l} & X \otimes Y \ar[r] & X \otimes_A Y.
    }
\end{align*}

The following result follows from Proposition 7.11.1 in \cite{EGNO15}.

\begin{prop}\label{prop:equ_mod_fun_bimod}
Let $(A, m_A, \iota_A)$, $(B, m_B, \iota_B)$ be algebras in $\CC$ such that $\bRMod_{\CC}(A)$, $\bRMod_{\CC}(B)$ are semisimple. Then the category $\Fun_{\CC|}(\bRMod_{\CC}(A), \bRMod_{\CC}(B))$ of left $\CC$-module functors is equivalent to $\bBMod_{\CC}(A|B)$. 

The equivalence is given by
    \begin{align*}
        X \mapsto - \otimes_A X: \bBMod_{\CC}(A|B) \to \Fun_{\CC|}\left (\bRMod_{\CC}(A), \bRMod_{\CC}(B) \right) \! . 
    \end{align*}
\end{prop}

\begin{defn}
A separable algebra $(A, m_A, \iota_A)$ in $\CC$ is called \textbf{indecomposable} if $\bRMod_{\CC}(A)$ is an indecomposable left $\CC$-module category, i.e. if it is not equivalent to a direct sum of non-zero left $\CC$-module categories. 
\end{defn}

\begin{defn}
An algebra $(A, m_A, \iota_A)$ is called \textbf{connected} (or \textbf{haploid}) if $\dim (\CC(\one, A)) = 1$, i.e. if $A$ is a simple object in $\bRMod_{\CC}(A)$.
\end{defn}

\begin{lem}\label{lem:con_alg_in_tensor}
    Let $\CC \simeq \oplus_{ij} \CC_{ij}$ be the decomposition as in Remark \ref{rem:standard_sol}. Then $(A, m_A, \iota_A)$ is a connected algebra in $\CC$ if and only if there exists exactly one $j \in \{1,\ldots,n\}$ such that $A=A_{jj}$ is a connected algebra contained in the tensor C$^*$-category $\CC_{jj}$ with tensor unit $\one_j$.  
\end{lem}

\begin{proof}
Recall $\one = \oplus_{i=1}^n \one_i$. By connectedness, there is only one $j$ such that $\CC(\one_j, A) \neq 0$, and $\dim(\CC(\one_j, A)) = 1$. Moreover, every $A_{kl}$ must be zero unless $k=l=j$.
\end{proof}

The following result is well-known, we sketch the proof for the reader's convenience:

\begin{lem}\label{lem:sep_alg_sum_ind}
    Let $(A, m, \iota)$ be a separable algebra in $\CC$. Then $A$ is a direct sum of indecomposable separable algebras. 
\end{lem}

\begin{proof}
    Note that $\bRMod_{\CC}(A)$ is indecomposable if and only if the identity functor $\mathrm{id} = - \otimes_A A$ associated with the trivial bimodule $A$ is a simple object in $\Fun_{\CC|}(\bRMod_{\CC}(A), \bRMod_{\CC}(A))$. By Proposition \ref{prop:equ_mod_fun_bimod}, 
\begin{align*} 
\bBMod_{\CC}(A|A)(A, A) \simeq \Fun_{\CC|}(\bRMod_{\CC}(A), \bRMod_{\CC}(A))(\mathrm{id},\mathrm{id}).
\end{align*}
Assume that $\dim (\bBMod_{\CC}(A|A)(A, A)) > 1$. Recall from Proposition \ref{prop:sep_alg_bimod_semi} that $\bBMod_{\CC}(A|A)$ is semisimple. Let $p$ be a non-trivial idempotent in $\bBMod_{\CC}(A|A)(A, A)$, i.e. $1_A - p \neq 0$, $p^2 = p$, and let $B$ be the image of $p$. Then $B$ is a separable algebra with multiplication and unit given by $v m (w \otimes w)$ and $v \iota$, where $v: A \to B$ and $w: B \to A$ are $A$-$A$-bimodule morphisms such that $v w = 1_{B}$ and $w v = p$. Note that $f: B \to B$ is a $B$-$B$-bimodule morphism with the previous algebra structure on $B$ if and only if $w f v: A \to A$ is an $A$-$A$-bimodule morphism. Thus $\dim (\bBMod_{\CC}(B|B)(B, B)) < \dim (\bBMod_{\CC}(A|A)(A, A))$. This implies that $A$ is a direct sum of indecomposable separable algebras. 
\end{proof}

\begin{rem}
If, in addition, the category $\CC$ is \textit{braided} and the separable algebra $(A, m, \iota)$ is \textit{commutative} in the sense of Definition 1.1 in \cite{KO02}, cf. Definition 4.20 in \cite{BKLR15}, then $\bBMod_{\CC}(A|A)$ and $\bRMod_{\CC}(A)$ can be identified. 
Hence, by the previous proof, $A$ is a direct sum of connected separable algebras, cf. Remark 3.2 in \cite{DMNO13}.
\end{rem}

\begin{lem}\label{lem:con_sep_alg_Frob}
    Let $(A,m, \iota)$ be a connected separable algebra in $\CC$. Then $A$ can be promoted to a Frobenius algebra.  
\end{lem}

\begin{proof}
By Lemma \ref{lem:con_alg_in_tensor}, we may assume that $\CC$ is a tensor C$^*$-category. Recall the conventions in Remark \ref{rem:standard_sol}. $\overline{A}$ is a right $A$-module with right $A$-action given by
\begin{align*}
    \overline{A} \otimes A \xrightarrow{1_{\overline{A} \otimes A} \otimes \overline{\gamma}_A} \overline{A} \otimes A \otimes A \otimes \overline{A} \xrightarrow{1_{\overline{A}} \otimes m \otimes 1_{\overline{A}}} \overline{A} \otimes A \otimes \overline{A} \xrightarrow{\gamma_A^* \otimes 1_{\overline{A}}} \overline{A}. 
\end{align*}
Let $f: A \to \overline{A}$ be the non-zero right $A$-module morphism defined by  
\begin{align*}
    f:= A \xrightarrow{1_A \otimes \overline{\gamma}_A} A \otimes A \otimes \overline{A} \xrightarrow{(\iota^*m) \otimes 1_{\overline{A}}} \overline{A}. 
\end{align*}
    Since $\bRMod_{\CC}(A)$ is semisimple by Proposition \ref{prop:sep_alg_bimod_semi}, $A$ is a simple right $A$-module by connectedness, and $d_A = d_{\overline{A}}$ (where $d_A$ is the scalar dimension \cite{GL19} of $A$ in $\CC$, or equivalently the dimension \cite{LoRo97} in $\CC_{jj}$, cf. Lemma \ref{lem:con_alg_in_tensor}), $f$ is invertible in $\CC$. Hence, by Lemma 3.7 in \cite{FRS02}, $A$ can be promoted to a Frobenius algebra.
\end{proof}

Let $(\CM, \odot)$ be a left $\CC$-module category. Then $\CM$ is said to be \textit{enriched} in $\CC$ if the functor $C \mapsto \CM(C \odot X, Y): \CC \to \vecc$ is \textit{representable} for every $X, Y \in \CM$, i.e. there exists an object $[X, Y]\in\CC$ such that
\begin{align*}
    \CM(- \odot X, Y) \simeq \CC(-, [X, Y]).
\end{align*}
The object $[X, Y]$ is called the \textit{internal hom} from $X$ to $Y$. In particular, $[X, -]: \CM \to \CC$ is the right adjoint of the functor $- \odot X: \CC \to \CM$. 

If $\CM =\bRMod_{\CC}(A)$, where $A$ is a separable algebra in $\CC$, then $\CM$ is enriched in $\CC$. More explicitly, the internal hom $[X, Y]$ is given by $\overline{X \otimes_A \overline{Y}}$.
We refer the reader to Section 7 in \cite{EGNO15} or Section 2 in \cite{KZ18} for basic facts about internal homs. 

\begin{lem}\label{lem:sep_alg_unitary_module}
Let $(A, m_A, \iota_A)$ be an indecomposable separable algebra in $\CC$. Then there exists a connected special C$^*$-Frobenius algebra $(B, m_B, \iota_B)$ in $\CC$ such that $\bRMod_{\CC}(A)$ and $\bRMod_{\CC}(B)$ are equivalent as left $\CC$-module categories. 

In particular, $\bRMod_{\CC}(A)$ is equivalent to a left C$^*$-module category over $\CC$.
\end{lem}

\begin{proof}
    Let $X$ be a non-zero simple object in $\bRMod_{\CC}(A)$. By Proposition \ref{prop:sep_alg_bimod_semi} and by the proof of Theorem 3.1 in \cite{Ost03} (cf. Theorem 2.1.7 in \cite{KZ18}), the internal hom $[X, X]$ in $\bRMod_{\CC}(A)$ is a connected (by the simplicity of $X$) algebra in $\CC$ such that $\bRMod_{\CC}(A)$ and $\bRMod_{\CC}([X, X])$ are equivalent. Note that $\bRMod_{\CC}(A)$ and $\bRMod_{\CC}([X,X])$ are both semisimple. Since 
\begin{align*}
     \Fun_{\CC|}(\bRMod_{\CC}([X, X]), \bRMod_{\CC}([X, X])) \simeq \Fun_{\CC|}(\bRMod_{\CC}(A), \bRMod_{\CC}(A)),   
\end{align*}
from Proposition \ref{prop:sep_alg_bimod_semi} and Proposition \ref{prop:equ_mod_fun_bimod} it follows that $A$ separable implies that $[X,X]$ is separable. By Lemma \ref{lem:con_sep_alg_Frob}, $[X,X]$ can be promoted to a connected Frobenius algebra. Then $[X,X]$ is isomorphic to a special C$^*$-Frobenius algebra $B$ in $\CC$ by Lemma \ref{lem:con_alg_in_tensor} and by Theorem 3.2, cf. Remark 3.3, in \cite{CGGH23}. We conclude that $\bRMod_{\CC}(A)$ is equivalent to $\bRMod_{\CC}(B)$. The latter is a left C$^*$-module category over $\CC$ by Proposition \ref{prop:special_Cstar_mod}.
\end{proof}

The following result is of independent interest and it should be compared with Lemma 2.18 in \cite{GS12} for $\CM =\bRMod_{\CC}(A)$, and Theorem A.1 in \cite{NY18}.

\begin{prop}\label{prop:internal_C_*_Frob_alg}
Let $(\CM, \odot)$ be an indecomposable left C$^*$-module over $\CC$ which is enriched in $\CC$. For every non-zero object $X$ in $\CM$, the internal hom $[X, X]$ is isomorphic (up to rescaling) to a special C$^*$-Frobenius algebra in $\CC$. 
\end{prop}

\begin{proof}
    By Proposition 2.3 in \cite{Reu23}, we may choose the right adjoint $[X, -]: \CM \to \CC$ of the $*$-functor $- \odot X: \CC \to \CM$ to be a $*$-functor. For every $C \in \CC$ and $Y \in \CM$, we treat $\CC(C, [X, Y])$ as the Hilbert space with inner product given by
\begin{align*}
    \braket{f_1}{f_2} := \tau \left (\gamma_{C}^*(1_{\overline{C}} \otimes f_1^*f_2) \gamma_C \right) \! ,
\end{align*}
where $\gamma_C$ and $\tau$ are defined in Remark \ref{rem:standard_sol}. Fix a faithful tracial state $\Tr$ on $\CM(X, X)$. We treat $\CM(C \odot X, Y)$ as the Hilbert space with inner product defined by
\begin{align*}
    \braket{g_1}{g_2} := \Tr \left( \left ((\gamma_C^* \otimes 1_X) (1_{\overline{C}} \odot g_1^*)  \right) \left ( (1_{\overline{C}} \odot g_2 ) (\gamma_C \otimes 1_X) \right) \right) \! .
\end{align*}
By the enrichment assumption, $\CC(-, [X, -])$ and $\CM(- \odot X, -)$ are equivalent bilinear $*$-functors $\CC^{\op} \times \CM \to \hilbc$, i.e. $\CC(f, [1_X, g])^* = \CC(f^*, [1_X, g^*])$ and $\CM(f \odot 1_X, g)^* = \CM(f^* \odot 1_X, g^*)$ for every $f \in \CC(C_2, C_1)$ and $g \in \CM(Y_1, Y_2)$. By considering the polar decomposition of natural isomorphisms, we may assume that the natural isomorphism $\CC(-, [X, -]) \simeq \CM(- \odot X, -)$ is componentwise unitary, i.e. $\CC(C, [X,Y]) \simeq \CM(C \odot X, Y)$ is unitary for every $C \in \CC$ and $Y \in \CM$.  

Note that $[X, -]$ is a left $\CC$-module functor with the $\CC$-module structure $\alpha_{C, Y}: C \otimes [X, Y] \xrightarrow{\sim} [X, C \odot Y]$ defined by the following natural isomorphism
    \begin{equation}\label{equ:internal_C_*_Frob_alg_1}
   \begin{aligned}
        \CC(B, C \otimes [X, Y]) &\xrightarrow{\sim} \CC( \overline{C} \otimes B, [X,Y]) \xrightarrow{\sim} \CM((\overline{C} \otimes B) \odot X, Y)\\
        &\xrightarrow{\sim} \CM(\overline{C} \odot (B \odot X), Y) \xrightarrow{\sim} \CM(B \odot X, C \odot Y) \xrightarrow{\sim} \CC(B, [X, C \odot Y]), 
   \end{aligned} 
\end{equation}
    where the first and fourth morphisms are induced by the solution of conjugate equation $(\gamma_C, \overline{\gamma}_C)$ and the third morphism is induced by the module structure of $\CM$ (see Section 7.12 in \cite{EGNO15}). By the fact that the natural isomorphism $\CC(-, [X, -]) \simeq \CM(- \odot X, -)$ is componentwise unitary, it is not hard to check the the natural isomorphism \eqref{equ:internal_C_*_Frob_alg_1} is unitary. Thus $\alpha_{C, Y}$ is unitary.

The evaluation $\ev_Y: [X, Y] \odot X \to Y$ is obtained as the image of $1_{[X, Y]}$ under the natural isomorphism $\CC([X, Y], [X, Y]) \simeq \CM([X, Y] \odot X, Y)$. Let $\ev_Y = h_Y u_Y$ be the polar decomposition of $\ev_Y$, where $h_Y := \sqrt{\ev_Y\ev_Y^*}$. Since $\alpha_{C,Y}$ is the unique morphism such that the following diagram commutes
\begin{align*}
    \xymatrix @R=0.2in{
        (C \otimes [X,Y]) \odot X \ar[d]^-{\alpha_{C, Y}} \ar[r]^-{\sim} & C \odot ([X,Y] \odot X) \ar[d]^-{1_C \odot \ev_Y}\\
        [X, C \odot Y] \odot X \ar[r]^-{\ev_{C \odot Y}} & C \odot Y,
    }
\end{align*}
    by the uniqueness of the polar decomposition, we have $1_C \odot h_Y = h_{C \odot Y}$. In particular, $h_Y: Y \to Y$ is a left $\CC$-module natural isomorphism of the identity functor $\Id_{\CM}$ to itself. Since $\CM$ is indecomposable, there exist $\lambda > 0$ such that $h_Y = \lambda 1_Y$ for every $Y$. Since the multiplication of $m: [X, X] \otimes [X, X] \to [X, X]$ is defined by 
\begin{align*}
    [X,X] \otimes [X,X] \xrightarrow{\alpha_{[X, X], X}} [X, [X,X] \odot X] \xrightarrow{[1_X, \ev_X]} [X,X],
\end{align*}
(see Section 3.2 in \cite{Ost03}) we have $mm^* = \lambda^2 1_{[X, X]}$. Hence $[X,X]$ can be rescaled to a special C$^*$-Frobenius algebra.
\end{proof}

Summing up, we can state and prove our main result:

\begin{thm}\label{thm:isomQsys_iff_sep}
    An algebra in a multitensor C$^*$-category $\CC$ is isomorphic to a special C$^*$-Frobenius algebra if and only if it is separable.
\end{thm}

\begin{proof}
    By Lemma \ref{lem:sep_alg_sum_ind}, we only need to show that every indecomposable separable algebra $(A, m_A, \iota_A)$ in $\CC$ is isomorphic to a special C$^*$-Frobenius algebra. Recall that $\bRMod_{\CC}(A)$ is equivalent to a left C$^*$-module category over $\CC$, denoted by $\CM$, by Lemma \ref{lem:sep_alg_unitary_module}. Let $F:\bRMod_{\CC}(A) \to \CM$ be the equivalence of left $\CC$-module categories. The algebra $A$ seen as an object of $\bRMod_{\CC}(A)$ equals $[A,A]$, see e.g. Remark 3.5 in \cite{Ost03}, hence it is isomorphic to $[F(A), F(A)]$. The latter is isomorphic to a special C$^*$-Frobenius algebra by Proposition \ref{prop:internal_C_*_Frob_alg}, hence $A$ is, and the proof is complete.
\end{proof}

For fusion C$^*$-categories $\CC$, the following is stated as Corollary 3.8 in \cite{CGGH23}, as a consequence of Theorem 3.2 therein.

\begin{cor}
    Let $\CM$ be a finite semisimple left module category over a multi-fusion C$^*$-category $\CC$. Then $\CM$ is equivalent to $\bRMod_{\CC}(A)$ for a special C$^*$-Frobenius algebra $A$. 
    
    Therefore, every finite semisimple left module category $\CM$ over a multi-fusion C$^*$-category $\CC$ admits a unique unitary structure (up to unitary module equivalence).
\end{cor}

\begin{proof}
    By Corollary 7.10.5 in \cite{EGNO15}, $\CM$ is equivalent to $\bRMod_{\CC}(B)$, where $B$ is an algebra in $\CC$. Since $\CM$ is semisimple, by Theorem 2.18 in \cite{ENO05}, we have that $\bBMod_{\CC}(B|B) \simeq \Fun_{\CC|}(\bRMod_{\CC}(B), \bRMod_{\CC}(B))$ is semisimple. Then $B$ is separable by Proposition \ref{prop:sep_alg_bimod_semi}, and  $\bRMod_{\CC}(B)$ is equivalent to $\bRMod_{\CC}(A)$ for a special C$^*$-Frobenius algebra $A$ by Theorem \ref{thm:isomQsys_iff_sep}.
    The uniqueness statement follows from Corollary 9 in \cite{Reu23}, see also Theorem 1 and Remark 4 therein.
\end{proof}

We conclude with an application of Theorem \ref{thm:isomQsys_iff_sep} which justifies Remark 4.2 in \cite{GiYu23}.
The \textit{idempotent completion} of a locally idempotent complete bicategory $\BB$, introduced in Definition A.5.1 in \cite{DR18-arxiv}, is the bicategory whose objects are \textit{separable} algebras in $\BB$, whose $1$-morphisms are bimodules, and whose $2$-morphisms are bimodule maps. By Proposition A.5.4 in \cite{DR18-arxiv}, there exists a canonical fully faithful bifunctor from $\BB$ into its idempotent completion. $\BB$ is called \textit{idempotent complete} if this bifunctor is a biequivalence. By combining the straightforward generalization of Theorem \ref{thm:isomQsys_iff_sep} to algebras in (rigid) semisimple C$^*$-bicategories and Lemma 4.1 in \cite{GiYu23}, we have the following result.

\begin{cor}
The rigid C$^*$-bicategory of finite direct sums of II$_1$ factors, finite Connes' bimodules and intertwiners is idempotent complete. 
\end{cor}

This result is also stated with a different but equivalent terminology in \cite{CHPJP22}. By Theorem \ref{thm:isomQsys_iff_sep}, at least for (rigid) semisimple C$^*$-bicategories, the terminology of \textit{Q-system completion} used in Definition 3.34 in \cite{CHPJP22} coincides with the previously mentioned idempotent completion of \cite{DR18-arxiv}.

\bigskip

\noindent\textbf{Acknowledgements.}
We thank Zheng Hao for insightful comments and Dave Penneys for informing us about a proof of Theorem \ref{thm:isomQsys_iff_sep} in the multifusion case that is to appear in \cite{CFHPS}. We also thank the referees for their suggestions and comments.

\bigskip


\def\cprime{$'$}

\end{document}